\documentclass[12pt,a4paper,reqno]{amsart}

\RequirePackage{amssymb, amsmath, amsfonts, latexsym, enumerate}
\RequirePackage[english]{babel}
\RequirePackage[OT1]{fontenc}
\RequirePackage[utf8]{inputenc}

\RequirePackage[colorlinks,linkcolor=blue,citecolor=blue,urlcolor=blue]{hyperref}

\usepackage[usenames,dvipsnames]{pstricks}
\usepackage{epsfig}

\newtheorem{thm}{Theorem}
\newtheorem{lemma}[thm]{Lemma}
\newtheorem{prop}[thm]{Proposition}
\newtheorem{cor}[thm]{Corollary}

\newtheorem{definition}[thm]{Definition}
\newtheorem{definitions}[thm]{Definitions}

\theoremstyle{remark}
\newtheorem{remark}[thm]{Remark}
\newtheorem{remarks}[thm]{Remarks}
\newtheorem{example}[thm]{Example}
\newtheorem{examples}[thm]{Examples}

\numberwithin{thm}{section}
\numberwithin{equation}{section}

\newcommand{\R}{\mathbb{R}}
\newcommand{\N}{\mathbb{N}}
\newcommand{\Rec}{\mathfrak{R}}

\newcommand{\Poi}{\mathbf{P}}
\newcommand{\Ca}{\mathbf{C}}

\newcommand{\1}{\mathbf{1}}

\newcommand{\boulette}[1]{$\bullet$\ Proof of #1.}
\newcommand{\Boulette}[1]{\par\noindent $\bullet$\ Proof of #1.}

\newcommand\haut[3]{#1^{#2#3}}
\newcommand\bas[3]{#1_{#2#3}}
\newcommand\pont[3]{#1^{#2#3}}

\renewcommand\AA{\mathcal{A}}

\newcommand\XX{\mathcal{X}}
\newcommand{\YY}{\mathcal{Y}}
\newcommand\XtX{\XX^2}
\newcommand\XXX[2]{X_{[#1,#2]}}

\newcommand\AAA[2]{\AA_{[#1,#2]}}

\newcommand\dPR{\frac{dP}{dR}}

\newcommand\pt{\textrm{-a.e.}}
\newcommand\pp{\textrm{-a.e.}}
\newcommand\OO{\Omega}

\newcommand\IX{\int_{\XX}}
\newcommand\IXX{\int_{\XX^2}}

\newcommand\pf{_{\#}}

\newcommand{\Ploop}{P _{ \mathrm{loop}}}
\newcommand{\Rp}{R ^{ \pi}}

\newcommand{\probmeas}{measure}

\begin{document}

\title[Reciprocal processes]{Reciprocal processes. A measure-theoretical point of view}

\author{Christian L\'eonard }
\address{Modal-X. Universit\'e Paris Ouest. B\^at.\! G, 200 av. de la R\'epublique. 92001 Nanterre, France}
\email{christian.leonard@u-paris10.fr}

\author{Sylvie R\oe lly}
\address{Institut f\"ur Mathematik der Universit\"at Potsdam. Am Neuen Palais 10. 14469 Potsdam, Germany}
\email{roelly@math.uni-potsdam.de}

\author{Jean-Claude Zambrini}
\address{GFM, Universidade de Lisboa, Av. Prof. Gama Pinto 2. 1649-003 Lisboa, Portugal}
\email{zambrini@cii.fc.ul.pt}

\thanks{First author was partly supported by the project GeMeCoD, ANR 2011 BS01 007 01, 
and last author by PTDC/MAT/120354/2010.
The authors are also thankful to the DFH for its support through the French-German Doktorandenkolleg CDFA 01-06.
}
 
 \subjclass[2010]{}

\begin{abstract}
This is a \emph{survey paper} about reciprocal processes. 
The bridges of a Markov process are also Markov. But an arbitrary mixture of these bridges fails to be Markov in general. 
However, it still enjoys the interesting properties of a {\em reciprocal process}.
\\
The structures of Markov and reciprocal processes are recalled with emphasis on their time-symmetries.  A review of the main properties of the  reciprocal processes is presented. Our measure-theoretical  approach allows for a unified treatment of the diffusion and  jump processes.  
Abstract results are illustrated by  several examples and counter-examples.       

\end{abstract}

\maketitle

\tableofcontents

\section*{Introduction} 
This is a survey paper about reciprocal processes.

The Markov property is a standard probabilistic  notion since its formalization at the beginning of the 20th century.
It was presented  for the first time in a time-symmetric way by Doob  \cite{Doob53}  in 1953, see \eqref{eq-01}. 
This remarkable point of view, which is often replaced by its time-asymmetric counterpart, was developed further by numerous  authors, 
see the references in  the monograph by Chung and Walsh \cite{ChW05} and also Wentzell \cite{Wen81}.

Two decades after  Markov, Bernstein \cite{Bern32} 
introduced  in the particular framework of diffusion processes the notion of \emph{reciprocal process}
\footnote{This terminology is due to Bernstein. They are also sometimes called Bernstein processes.} 
which, by its very definition,  is 
stable with respect to time reversal, see \eqref{eq-rec}.  
During his talk at the International Congress in Zürich, he qualified   the  dynamics of reciprocal processes as  {\it stochastiquement parfaite}\footnote{stochastically perfect.}. Bernstein's motivation for introducing these time-symmetric notion of random process is rooted into an outstandingly  insightful article by Schrödinger  \cite{Sch31} entitled {\it ``\"Uber die Umkehrung der Naturgesetze''} published in 1931.

The particular case of Gaussian reciprocal processes was extensively treated by Jamison \cite{Jam70}. This was also undertaken by Chay \cite{Chay72} who named these processes ``quasi-Markov random fields" since they are  random fields defined on a compact time interval. See further comments and references at the beginning of Section \ref{sec-rpm} and Subsection \ref{sec:recClass}. Later, in the excellent seminal  paper \cite{Jam74}  entitled  ``Reciprocal processes'', Jamison provided a rigorous definition of these processes and derived  their main properties in an abstract setting.  
This led to many studies in specific contexts, such as diffusion processes  \cite{TZ97b,TZ97a,CWZ,Th02}, Lévy processes \cite{PZ04} or pure jump processes  \cite{CLMR14,CL14}.

This review paper revisits \cite{Jam74} and provides some new results. We  present a unifying measure-theoretical approach to reciprocal processes. Unlike Jamison's approach, which is based on the usual forward filtration, our  point of view allows for a genuinely time-symmetric treatment. 
We look at Markov and reciprocal processes as \emph{path measures}, i.e.\ measures on the path space. Consequently, instead of considering Markov and reciprocal random processes as usual, we consider  Markov and reciprocal path measures.

 Theorem \ref{res-20} illustrates our perspective. Its statement gives a  characterization of the reciprocal
measures which are dominated by a given  Markov measure in terms of time-symmetric versions of Doob's $h$-transforms. 
We also illustrate our abstract results with several examples and counter-examples.       

\subsubsection*{A possible extension}
We   focus onto  \emph{probability} path measures. As a consequence, in the present paper we drop the word probability: any path probability measure, Markov probability measure or reciprocal probability measure  is simply called a path measure, a Markov measure or a reciprocal measure.\\
Our results can easily be extended to 
$\sigma$-finite path measures, e.g.\  processes admitting an unbounded  measure as their initial law. 
For further  detail about this generalized 
framework, see  \cite{Leo12b}.

\subsection*{Outline of the paper}
Section \ref{sec-Markov} is devoted to the structure of  Markov \probmeas s and to their time-symmetries. 
Reciprocal measures are introduced at Section \ref{sec-rpm} and their relationship with Markov measures is investigated. 
At Section \ref{sec-ent}, we sketch the tight connection
between reciprocal classes and some specific entropy minimization problems. 
So doing, we step back to Schrödinger's way of looking at  some statistical
physics problems with a time-symmetric viewpoint.

\subsection*{Notation}

We consider the  set $\Omega=D([0,1],\XX)\subset \XX^{[0,1]}$ of càdlàg paths defined on the finite time interval 
$[0,1]$ with state space $\XX$, which is assumed to be Polish and equipped  with its Borel $\sigma$-algebra.  
As usual $\Omega$ is endowed with the canonical filtration  $\AA$ generated by the  \emph{canonical process} $X=(X_t)_{t\in[0,1]}$ : 
$$
 X_t(\omega):=\omega_t, \quad \omega=(\omega_s)_{s\in[0,1]}\in\Omega, \, t\in[0,1].
$$
For any subset  $\mathcal{S}\subset [0,1]$ and for any \probmeas\   $P$ on  $\Omega$ one denotes
\begin{itemize}
    \item $X_\mathcal{S}=(X_s)_{s\in\mathcal{S}}$ the 
    canonical process restricted to $\mathcal{S},$
    \item $\AA_\mathcal{S}=\sigma(X_s;s\in\mathcal{S})$ the   $\sigma$-algebra  of the events observed during $\mathcal{S},$
    \item $P_\mathcal{S}=(X_\mathcal{S})\pf P$ the restriction of $P$
    to $\Omega_\mathcal{S}:=X_\mathcal{S}(\Omega)$.
\end{itemize}
We have denoted $f\pf m:=m\circ f ^{ -1}$ the image of the measure $m$ with respect to the measurable mapping $f$.\\
 For   $\mathcal{S}= [s,u]\subset [0,1]$ we use the particular notations:
\begin{itemize}
\item   $\XXX su:=(X_t; s\le t\le u)$
    \item $\AAA su:=\sigma(\XXX su),$ the   $\sigma$-algebra generated by the events that occurred between time  $s$ and time $u$
    \item $\AA_s:=\sigma(X_s),$ the   $\sigma$-algebra generated by the events that occur at time  $s$
    \item  $P_s:=(X_s)\pf P$ is the  projection of $P$ at time  $s$
    \item $\bas Psu:=(X_s,X_u)\pf P$ is the marginal law of  $P$ at times  $s$ and $u$ simultaneously  \\
    ($P_{01}$ is therefore the endpoint marginal  law of the process)
    \item $P_{[s,u]}:= (X_{[s,u]})\pf P$ is the projection of $P$ on the time interval $[s,u]$
\end{itemize}
The \probmeas\ 
$$
P^*=(X^*)\pf P
$$
 is the law under  $P$ of the càdlàg transform $X^*$ of the time reversed canonical process.

\section{Time-symmetry of Markov \probmeas s} \label{sec-Markov}

We present  structural properties of both Markov probability  measures and  their bridges.
Emphasis is put on their time-symmetry which has already been studied in specific frameworks, see for instance
\cite{ChW05}.

\subsection{Definition and basic properties } 

Let us begin with the symmetric definition of the Markov property.
\begin{definition}[Markov \probmeas] \label{def-PMark} 
A probability measure  $P$ on  $\Omega$ is said to be  \emph{Markov} 
 if for any $t\in[0,1]$ and for any events $A\in\AAA 0t, B\in \AAA
    t1$
    \begin{equation}\label{eq-01}
    P(A\cap B\mid X_t)=P(A\mid X_t)P(B\mid X_t), \quad P\pp 
    \end{equation}
\end{definition}
This means that, knowing the present state  $X_t$, the future and  past   informations  $\AAA
    t1$  and  $\AAA 0t$, are  $P$-independent.  

In   Theorem \ref{res-02} below, we recall equivalent descriptions of the Markov property. 
In particular, the standard identity (2)  states that a Markov process forgets its past history. 

\begin{thm}\label{res-02} Let $P$ be a probability measure   on $\Omega$.
Then the following are equivalent:
\begin{itemize}
    \item[(1)] 
\,The \probmeas\   $P$ is Markov.
     \item[(1*)]  
\, The time-reversed \probmeas\  $P^*$ is Markov.
\vspace{3mm}

    \item[(2)] 
For all $0\le t\le1$ and all sets $B\in \AAA t1$,    
\begin{equation*}
P(B\mid\XXX 0t)=P(B\mid X_t), \quad  P\pp
\end{equation*}
	\item[(2*)] 
For all  $0\le t\le1$ and all  sets $A\in\AAA 0t,$
    \begin{equation*}
     P(A\mid\XXX t1)=P(A\mid X_t), \quad  P\pp
    \end{equation*}


    \item[(3)] 
\, For all $0\le s\le u\le1$  and all sets $A\in\AAA 0s, C\in\AAA
    u1$
    $$
    P(A\cap C\mid\XXX su)=P(A\mid X_s)P(C\mid X_u), \quad P\pp
    $$
\end{itemize}
\end{thm}

\begin{proof}
Let us prove  $(3)\Rightarrow (1)\Rightarrow (2)\Rightarrow (3) .$ 

\Boulette{$(3)\Rightarrow (1)$} It is clear by taking  $s=u$.

\Boulette{$(2)\Rightarrow (3)$} For all sets 
 $A\in\AAA 0s$ and $C\in\AAA u1$ and all sets 
$B\in \AAA su,$ the equality 
$$
P(A\cap B \cap C)
=E[\1_{B}
P(A\cap C\mid\XXX su)]
$$ 
holds. On the other hand,
\begin{eqnarray*}
  P(A\cap B \cap C)
  &=& E [P(A\cap B \cap C \mid\XXX 0u)] \\
  &=& E[\1_{A}\1_{B} P(C \mid \XXX 0u)] \\
  &=&  E[\1_{A}\1_{B} P(C \mid \XXX su)] \\
  &=& E[\1_{B} P(A\mid \XXX su)P(C\mid \XXX su)]
\end{eqnarray*}
where property  (2) is used in the third equality. Therefore
$$
P(A\cap C \mid\XXX su)=P(A\mid \XXX su)P(C\mid \XXX su) .
$$

\Boulette{$(1)\Rightarrow (2)$} 
It is based on standard properties of conditional independence. Nevertheless, for the sake of completeness,
we sketch it.
Let us show that if   \eqref{eq-01} is satisfied then  $P$ forgets its past history.
Let $A\in\AAA 0t$ and
$B\in\AAA t1$ be some events.
$$
E[\1_{A} P(B\mid\XXX 0t)]=P(A\cap B)
=E(P(A \cap B\mid X_t)) =E[P(A\mid X_t)P(B\mid X_t)].
$$ 
On the other hand, 
$$E[\1_{A} P(B\mid X_t)]=E[P(A\mid X_t)P(B\mid X_t)].$$
One obtains for any set $A\in\AAA 0t$,
$E[\1_{A}P(B\mid\XXX 0t)]=E[\1_{A} P(B\mid X_t)],$ which implies that
$P(B\mid\XXX 0t)=P(B\mid X_t).$ This completes the proof of $(1)\Rightarrow (2)$, together with the proof of $(1)\Leftrightarrow(2)\Leftrightarrow(3).$
\\
Eventually the symmetry of the formulation of  (3) leads to the equivalence between
 (2) and ($1^*$). Assertion ($2^*$) corresponds to  (2) applied to $P^*$.
\end{proof}

The  first proof of $(1)\Leftrightarrow(2)$ appears in the monograph by Doob \cite[Eq.\,(6.8) \& (6.8')]{Doob53}.
 Then,  Dynkin \cite{Dyn61} and Chung \cite[Thm.\,9.2.4]{Ch68} 
took it over. Meyer already remarked in  \cite{Mey67} that  {\it the Markov property is invariant under time reversal}.

Identity (2), called \textit{left-sided Markov property}, is often used as the definition of the Markov property.
It may create the inaccurate delusion (frequent in the context of statistical physics) that the Markov property is time-asymmetric. 

Since each Markov process can be defined via its forward and backward transition probability kernels, 
we recall how to construct them in a \emph{symmetric} way.

\begin{definitions} Let  $P$ be a Markov \probmeas.
\begin{enumerate}
    \item The \emph{forward transition probability kernel} associated with $P$
is the family of conditional \probmeas s
$\big(p(s,x;t,\cdot); 0\le s\le t\le1, x\in\XX \big)$ defined for 
any $0\le s\le t\le1$, and $P_s$-almost all $x$, by
\begin{equation*}
    p(s,x;t,dy)=P(X_t\in dy\mid X_s=x).
\end{equation*}
    \item  The \emph{backward transition probability kernel} associated with $P$
is the family of conditional \probmeas s
$\big(p^*(s,\cdot;t,y); 0\le s\le t\le1, y\in\XX \big)$ defined for 
any  $0\le s\le t\le1$, and $P_t$-almost all $y$, by
\begin{equation*}
    p^*(s,dx;t,y):=P(X_s\in dx\mid X_t=y).
\end{equation*}
\end{enumerate}
\end{definitions}

Since these kernels satisfy the celebrated  Chapman-Kolmogorov relations 
\begin{eqnarray}
  \forall 0\le s\le t\le u\le1,&& \nonumber\\ 
p(s,x;u,\cdot)&=&
\int_{\XX} p(s,x;t,dy)p(t,y;u,\cdot) \ \textrm{ for }P_s\textrm{-a.a. } x \label{eq-11f}
 \\
p^*(s,\cdot;u,z)&=& 
\int_{\XX} p^*(s,\cdot;t,y)p^*(t,dy;u,z) \ \textrm{ for }P_u\textrm{-a.a. } z, \label{eq-11b}
\end{eqnarray}
one can construct the \probmeas\  $P$  in the following way.

\begin{prop}\label{res-16}
The Markov \probmeas\  $P$ is uniquely determined by 
one time marginal $P_u$ at some time $u \in [0,1]$, its 
forward transition probability kernels starting from  time  $u$, $\big(p(s,x;t,\cdot); u\le s\le
t\le1, x\in\XX\big)$ and the backward transition probability kernels since time $u$, $\big(p^*(s,\cdot;t,y); 0\le
s\le t\le u, y\in\XX\big).$ \\
Indeed, for any $0\le s_1\le\dots
s_k\le u\le t_1\le\dots t_l\le1$ and $k,l\ge1,$ the finite dimensional projection of $P$ are given by
\begin{equation*}
    P_{s_1,\dots, s_k, u, t_1,\dots, t_l}=
    p^*_{s_1;s_2}\otimes\cdots\otimes
    p^*_{s_{k};u}\otimes P_u\otimes p_{u;t_1}\otimes\cdots\otimes
    p_{t_{l-1};t_l} .
\end{equation*}
where we used the following intuitive notation
$$
P_u\otimes p_{u;t} (dx,dy):= P_u(dx) p(u,x;t,dy).
$$ 
\end{prop}

\subsection{Path measures dominated by a Markov \probmeas} 

We  consider the problem of knowing  whether a path \probmeas\ which is dominated by  a 
reference Markov \probmeas\ inherits its Markov property.
The following result states a criterion in terms of  the multiplicative structure of the Radon-Nikodym derivative on the path space. 
This question was posed in the general background of continuous time Markov fields in the 70' and 
solved by Dang Ngoc and Yor in \cite[Prop.\,4.3(b)]{DY78}. 
Some abstract variant of this result 
also appeared in  \cite[Thm.\,3.6]{vPvS85}.
In the simpler framework of processes indexed by a discrete time, we refer the reader to Georgii's extended monograph \cite[Ch.\,10, 11]{Ge11}. 

\begin{thm}\label{res-19}
Let $R$ be a reference Markov \probmeas\   and let  $P\ll R$,  a probability measure  dominated by $R$. Then the following statements are equivalent.
\begin{enumerate}
    \item[(1)] The \probmeas\  $P$ is Markov.
    \item[(2)] For any time  $t\in [0,1]$, the Radon-Nikodym derivative of $P$ with respect to  $R$ factorizes in the following way:
 \begin{equation}\label{eq-39}
    \dPR=\alpha_t \, \beta_t, \quad R\pp
\end{equation}
where
$\alpha_t$ and $\beta_t$ are respectively nonnegative 
$\AAA 0t$-measurable and  $\AAA t1$-measurable functions.  
\end{enumerate}
\end{thm}

\begin{proof}
For an alternate proof, see \cite[Prop.\,4.3]{DY78}.
\smallskip\Boulette{$(2)\Rightarrow(1)$} Take two events,  $ A\in\AAA 0t$ and $ B\in\AAA t1$.
In terms of Definition \ref{def-PMark}, we have to show that 
\begin{equation}\label{eq-14}
    P( A\cap B \mid X_t)=P( A\mid X_t)P( B\mid  X_t),\quad P\pp.
\end{equation}
To this aim, note that although the product  $\alpha_t \beta_t $ is $R$-integrable, it is not clear why $\alpha_t$ or $\beta_t$ 
should be separately integrable. To prove this required integrability, one may use the following  lemma of integration theory which assures 
the $R(\cdot\mid X_t)$-integrability of the functions  $\alpha_t$ and $\beta_t,$ $P\pp$

\begin{lemma}\label{res-24} Assume that statement  (2) of the above  theorem holds true. Then,  the functions  $\alpha_t$ and $\beta_t$ are
$R(\cdot\mid X_t)$-integrable $P\pp$
and
$$
\left\{
\begin{array}{l}
  0< E_R(\alpha_t \beta_t \mid X_t)=E_R(\alpha_t \mid X_t)E_R(\beta_t \mid   X_t),  \quad P\pp\\
 0\le E_R(\alpha_t \beta_t \mid X_t)
    =\1_{\{E_R(\alpha_t \mid X_t) E_R(\beta_t \mid X_t)<+\infty\}}
    E_R(\alpha_t \mid X_t)E_R(\beta_t \mid X_t),\quad R\pp\\
\end{array}
    \right.
$$
 \end{lemma}

\begin{proof}
See \cite[\S\,3]{Leo12b}.
\end{proof}

 Lemma \ref{res-24} leads  to
\begin{equation*}
  P(A \cap B\mid X_t)
  = \frac{E_R(\alpha_t \beta_t \, \1_A \1_B\mid X_t)}{E_R(\alpha_t \beta_t\mid X_t)}
  = \frac{E_R(\alpha_t \,\1_A\mid X_t)}{E_R(\alpha_t\mid X_t)}\frac{E_R(\beta_t\, \1_B\mid X_t)}{E_R(\beta_t\mid
  X_t)} , \ P\pp 
\end{equation*}
Choosing $ A=\Omega$ or $ B=\Omega$ in this formula, we obtain 
\begin{eqnarray*}
P( B\mid X_t) &=& E_R(\beta_t \,\1_B\mid X_t)/E_R(\beta_t \mid
X_t),\\
 P( A\mid X_t) &=& E_R(\alpha_t \,\1_A\mid X_t)/E_R(\alpha_t\mid
 X_t).
\end{eqnarray*} 
This completes the proof of \eqref{eq-14}. 

\smallskip\Boulette{$(1)\Rightarrow(2)$} Take a Markov \probmeas\  $P$ with derivative $Z$  with respect to $R:$ $dP=Z\, dR$. 
We denote by
$$
Z_t:=E_R(Z\mid\XXX 0t),\, 
Z^*_t:=E_R(Z\mid\XXX t1) \textrm{ and } \zeta_t(z):=E_R(Z\mid
X_t=z)=\frac{dP_t}{dR_t}(z).
$$ 
Remark that the last equality implies that $ \zeta_t(X_t)>0,\ P\pp$,    
\begin{equation}\label{eq-50}
   \zeta_t(X_t) = E_R(Z_t\mid X_t)=E_R(Z_t^*\mid
X_t), \quad R\pp
\end{equation}
and that  $ \zeta_t(X_t)$ is $R$-integrable.\\
Fix three bounded nonnegative functions  $f,g,h$ that are respectively $\AAA 0t$, $\AA_t$ and
$\AAA t1$ measurable. One obtains
\begin{eqnarray*}
  E_P(fgh)
  &\overset{\textrm{(i)}}=& E_P\left[E_P(f\mid X_t)\ g\ E_P(h\mid X_t)\right] \\
  &\overset{\textrm{(ii)}}=& E_P\left[
  \frac{E_R(fZ_t\mid X_t)}{E_R(Z_t\mid X_t)}
        \ g\ \frac{E_R(hZ^*_t\mid X_t)}{E_R(Z_t^*\mid X_t)}\right]\\
  &\overset{\textrm{(iii)}}=& E_P\left[g \frac{E_R(fhZ_tZ^*_t\mid X_t)}{\zeta_t(X_t)^2}\right] \\
  &\overset{\textrm{(iv)}}=& E_P[g E_{\widetilde{P}}(fh\mid X_t)]
\end{eqnarray*}
where we  successively used in (i):  the Markov property of  $P$,
in (iii): identity \eqref{eq-50} and the Markov property of $R$
and in (iv),
we introduce the \probmeas\ 
\begin{equation}\label{eq-54}
\widetilde{P}:=\1_{\{\zeta_t(X_t)>0\}}\frac{Z_tZ_t^*}{\zeta_t(X_t)}\,
R.
\end{equation}
From all these identities one deduces that
\begin{equation}\label{eq-51}
    P(\cdot\mid X_t)=\widetilde{P}(\cdot\mid X_t), \quad P\pp
\end{equation}
Define
$$
\left\{
\begin{array}{lcl}
  \alpha_t&=&\1_{\{\zeta_t(X_t)>0\}}\,Z_t/\zeta_t(X_t) \\
  \beta_t&=& Z_t^*.
\end{array}\right.
$$
Therefore   \eqref{eq-54} becomes
\begin{equation}\label{eq-55}
    \widetilde{P}=\alpha_t\beta_t\,R 
\end{equation}
and
$$
  E_R(\alpha_t\mid X_t)=\1_{\{\zeta_t(X_t)>0\}} \textrm{ and } E_R(\beta_t\mid X_t)=\zeta_t(X_t) .
$$
In order to identify  $P$ with $\widetilde{P}$, since (\ref{eq-51}) is satisfied, it is enough to show that their marginals at time $t$ are the same. Let us prove it. 
\begin{eqnarray*}
  \widetilde{P}_t(dz)
  &=& E_R\left(\alpha_t\beta_t\mid X_t=z\right)\,R_t(dz)  \\
  &\overset{(i)}=& E_R\left(\alpha_t\mid X_t=z\right)E_R\left(\beta_t\mid X_t=z\right)\,R_t(dz)  \\
  &=& \zeta_t(z) \,R_t(dz) =  P_t(dz)
\end{eqnarray*}
where the Markov property of $R$ is used at ({\it i}). This fact, together with 
 \eqref{eq-51}, implies the equality $P=\widetilde{P}.$ Eventually, since 
 $Z_t$ is $\AAA 0t$-measurable and $Z_t^*$ is $\AAA t1$-measurable, 
$\alpha_t$ and $\beta_t$ are respectively $\AAA 0t$  and  $\AAA t1$-measurable functions.
\end{proof}

\begin{example}\label{ex-2.16}
In the extreme case where   $\alpha_t$ is  $\AA_0$-measurable (that is $\alpha_t=f_0(X_0) $) and $\beta_t$  is $\AA_1$-measurable 
(that is $\beta_t=g_1(X_1)$), 
one obtains from the above theorem that any \probmeas\  $P$ of the form 
\begin{equation}\label{eq-fg}
P = f_0(X_0)g_1(X_1) \, R
\end{equation}
is Markov. This was remarked in \cite[Thm.\,4.1]{DY78} (Eq.\,(1) applied to $a=0$ and $b=1$).
\end{example}
In Theorem \ref{res-20} we will see that, under some restrictions on $R$, the \probmeas s of the form \eqref{eq-fg} 
are the only ones which are Markov in the class of all  \probmeas s of the form $P = h(X_0,X_1) \, R$.

\subsection{A fundamental example: bridges of a Markov measure} \label{subsection1.3} 

Since we are interested in a time-symmetric description of path measures, 
it is reasonable to disintegrate them along their endpoint (initial and final) values. Any probability measure 
$P$ on $ \Omega $ is a mixture of  \probmeas s pinned at both times $t=0$ and  $t=1,$ i.e.\  a mixture of its own bridges:
\begin{equation}\label{eq-34}
    P=\IXX P(\cdot\mid X_0=x,X_1=y)\,\bas P01(dxdy).
\end{equation}
Since $\XX^2$ is Polish,  this disintegration is meaningful. 
Note however that the bridge  $P( \cdot \mid X_0,X_1)$ is a priori only defined $P$-a.s. 

To simplify the presentation of our results, we will consider path measures $P$ whose 
 bridges can be constructed for \emph{all}  $x,y \in \XX$ as a regular version of the family of conditional laws 
($P(\cdot\mid X_0=x,X_1=y), x,y \in \XX$) and denote them by   $(\pont Pxy)_{ x,y \in \XX}$. 
See for instance \cite[Ch.\,6]{Kal81} for a precise definition of a regular 
conditional distribution. \\
Remark that it is not easy to construct such an everywhere-defined version in a general non-Markov setting 
but this is done in several relevant situations: When
$P$ is a Lévy process -- see \cite{Kal81} and \cite[Prop.\,3.1]{PZ04},
  a right process -- see \cite{FPY92}, or a Feller process -- see  \cite{ChU11}.\\
We recall the important property that  pinning preserves the Markov property.
\begin{prop}\label{res-11bis} 
Let $P$ be a Markov \probmeas\   whose bridges are  defined everywhere.
Then, for any $x,y \in \XX$, the bridge $\pont Pxy$ is also Markov. 
\end{prop}
\begin{proof}
Let $P$ be a Markov \probmeas,  $t$ be  a time in  $[0,1]$,  $A\in\AAA 0t$ and $ B\in\AAA t1$ be two events. 
We have
\begin{equation} \label{eq-pont}
 P(A\cap B\mid X_0,X_t,X_1)=P(A\mid X_0,X_t)P(B\mid X_0,X_t,X_1),\quad P\pp 
\end{equation}
Indeed,
\begin{eqnarray*}
  P(A\cap B \mid X_0,X_t,X_1)
  &=& E [P(A\cap B \mid X_0, \XXX t1)\mid X_0,X_t,X_1] \\
  &=& E[\1_B P(A\mid X_0, \XXX t1) \mid X_0,X_t,X_1] \\
  &=&  E[\1_B P(A\mid X_0, X_t) \mid X_0,X_t,X_1] \\
  &=& P(A\mid X_0, X_t)P(B \mid X_0,X_t,X_1).
\end{eqnarray*}
Moreover, by  Theorem \ref{res-02}-(2*),  $P(A\mid X_0,X_t)= P(A\mid X_0,X_t, X_1)$. Therefore 
$$
P^{X_0,X_1}(A\cap B\mid X_t)=P^{X_0,X_1}(A\mid
X_t)P^{X_0,X_1}(B\mid X_t), \quad P\pp 
$$ 
which characterizes the Markov property of 
 every bridge  $\pont Pxy$ via \eqref{eq-01}.
\end{proof}

In the rest of the section, we will work in the following framework. 

\noindent {\bf Assumptions (A).}\label{AA}\ 
There exists a reference Markov \probmeas\  $R$ satisfying the following requirements.\\
$(\textrm{A}_1)\quad R$ admits a family of bridges which can be defined everywhere\\
$(\textrm{A}_2) \quad$The transition probability kernels of $R$ admit a density, denoted by $r$, with respect to 
some $\sigma$-finite positive measure $m$ on $\XX$:
For all  $ 0\le
s<t\le1,$
\begin{eqnarray*}
r(s,x;t,y)&:=& \frac{dr(s,x;t,\cdot)}{dm}(y) \textrm{ for } R_s\otimes m\textrm{-a.e. }(x,y) \\
\textrm{and}\quad
 r^*(s,x;t,y)&:=& \frac{dr^*(s,\cdot;t,y)}{dm}(x)  \textrm{ for } m\otimes R_t\textrm{-a.e. }(x,y) .
\end{eqnarray*}
Therefore
$
R_0(dx)=\int r^*(0,x;1,y) R_1(dy) \, m(dx) =: r_0(x)\, m(dx)$ 
and similarly
$
R_1(dy) =: r_1(y)\, m(dy) .
$
This leads to 
$$
\bas
R01(dxdy)= r_0(x)\, m(dx) r(0,x;1,y)m(dy) =r_1(y)m(dy) r^*(0,x;1,y) m(dx),
$$
in such a way  that the function $c$ defined for almost every $x$ and $y$ by
\begin{equation} \label{eq:densityjointmarg}
 c(x,y):=r_0(x)r(0,x;1,y)=r_1(y) r^*(0,x;1,y)
\end{equation}
is the  density of the joint marginal $\bas R01(dxdy) $  with respect to  $m \otimes m $.\\

Remark that  $(\textrm{A}_2)$ is not always satisfied.  For instance, let us consider a Poisson process $R$ with a random  initial law that admits a density on  $\R$.
At any time $s$, its marginal law also admits a density. But the support of the measure $r(s,x;t,dy)$ is discrete and equal to $x + \N$. Therefore there does not
exist any measure $m$ such that for a.e. $x$, $r(s,x;t,dy)\ll  m (dy)$.
We will see at Example \ref{ex-ponts}(ii) how to circumvent this obstacle.\\

Let us recall the general structural relation between the path measure $R$ and its bridges. 
In general, the bridges are not globally absolutely continuous with respect to $R$, but they are locally
absolutely continuous with respect to $R$ in restriction to time interval $[s,t] \subset (0,1)$ away from the terminal times 0 and 1.

\begin{thm} Consider a Markov \probmeas\  $R$ 
satisfying  Assumptions {\em (A)}.
For all  $0<s\le t<1$ and all $x,y \in \XX$,
the bridge  $(\pont Rxy)_{[s,t]}$
of $R$  restricted to  $\AAA st$ is dominated by $R_{[s,t]}$. Its density is given by
\begin{equation}\label{eq-24}
    (\pont Rxy)_{[s,t]}=\frac{r^*(0,x;s,X_s)\, 
    r(t,X_t;1,y)}{c(x,y)} \, R_{[s,t]} \, ,
\end{equation}
where the function $c$ is defined by \eqref{eq:densityjointmarg}.
\end{thm}

\begin{proof}
Let us show that
\begin{equation}\label{eq-28}
c(x,y)=0 \Rightarrow r^*(0,x;s,z)r(t,z';1,y)=0,\quad \forall
(z,z'),\ \bas Rst\pp
\end{equation}
On the one hand,
$$
\bas R01 (dx dy)= c(x,y)  m(dx) m(dy)
$$
and on the other hand, following Proposition \ref{res-16},
\begin{eqnarray*}
\bas R01 (dx dy)&=& \IXX R_{0,s,t,1}(dx,dz,dz',dy)\\
&=& \IXX  r^*(0,dx;s,z)R_s(dz)  r(s,z;t,dz') r(t,z';1,dy) \\
&=& \IXX  r^*(0,x;s,z) r(s,z;t,z') r(t,z';1,y)R_s(dz)m(dz') \,  m(dx) m(dy).
\end{eqnarray*}
Then
$$
c(x,y) =  \IXX  r^*(0,x;s,z) r(s,z;t,z') r(t,z';1,y)R_s(dz)m(dz')
$$
and  (\ref{eq-28}) holds. This allows us not to bother about dividing by zero.

For $\bas Rst$-a.e. $(z,z'),$ 
the \probmeas\  $r^*(0,dx;s,z)r(t,z';1,dy)$ is dominated by $\bas R01(dxdy)$ and it satisfies 
    \begin{equation}\label{eq-27bis}
r^*(0,dx;s,z)r(t,z';1,dy)=\frac{r^*(0,x;s,z)r(t,z';1,y)}{c(x,y)}\,
\bas  R01(dxdy).
\end{equation}
Take two bounded measurable functions $f,g$ and an event 
$B\in\AAA st$.
Thus,
\begin{eqnarray*}
  &&E_R[f(X_0)\ \1_B \ g(X_1)]\\
  &=& E_R\left[ \1_B \ E_R (f(X_0)\mid\XXX st) \ E_R(g(X_1)\mid\XXX st)\right] \\
  &=& E_R\left[  \1_B \ E_R(f(X_0)\mid X_s) \ E_R(g(X_1)\mid X_t)\right]\\
  &=& E_R\left[ \1_B \IX f(x) \,r^*(0,dx;s,X_s) \ \IX
  g(y)\,r(t,X_t;1,dy)\right]\\
  &=& E_R\left[\1_B \IXX f(x)g(y)r^*(0,dx;s,X_s)r(t,X_t;1,dy)\right]\\
 &\overset{\checkmark}=&E_R\left[\1_B \IXX f(x)\frac{r^*(0,x;s,X_s)r(t,X_t;1,y)}{c(x,y)}g(y)\,\bas
  R01(dxdy)\right]\\
  &=&\IXX f(x)E_R\left[\1_B \frac{r^*(0,x;s,X_s)\,r(t,X_t;1,y)}{c(x,y)}\right]g(y)\,\bas
  R01(dxdy),
\end{eqnarray*}
where we used  \eqref{eq-27bis} at the marked equality. This proves  \eqref{eq-24}. 
\end{proof}

\begin{cor}[Decomposition of a bridge]
Introducing $f_s(z):=r^*(0,x;s,z)$ and $g_t(z')=:c(x,y)^{-1}r(t,z';1,y)$,  \eqref{eq-24} becomes
\begin{equation}\label{eq-24bis}
   (\pont Rxy)_{[s,t]}=f_s(X_s)\, 
    g_t(X_t) \, R_{[s,t]}.
\end{equation}
 In particular, at each time  $t \in (0,1)$, the one dimensional marginal of the bridge $\pont Rxy$ is dominated by the 
 marginal $R_t$ of the Markov \probmeas\  $R$. It satisfies
 $$
 \pont Rxy_{t}=f_t(X_t)\, g_t(X_t) \, R_{t}.
$$
\end{cor}

One interprets  \eqref{eq-24bis} as a generalization of (\ref{eq-fg}) on the time interval $[s,t]$:
the density of the bridge decomposes into a product of functions of the process at boundary times  $s$ and $t$. 
This  ensures its Markov property.

Naturally, both forward and backward dynamics of the bridge are directly related to the dynamics of the reference
process with free boundary conditions. 
\begin{prop} Let  $R$ be a Markov \probmeas.
\begin{enumerate}
    \item[(1)] For any time $ 0<t<1$ and for any $(x,y),$ the bridge  $\pont Rxy$ of $R$,
  restricted to  $\AAA 0t$ is given by 
\begin{equation}\label{eq-21f}
    (\pont Rxy)_{[0,t]}=\frac{r(t,X_t;1,y)}{r(0,x;1,y)}\, R_{[0,t]}(\cdot\mid X_0=x) .
\end{equation}
    \item[(2)] Analogously, for any time  $0<s<1$ and for any 
$(x,y),$ the bridge $\pont Rxy$ of $R$ restricted to 
$\AAA s1$ is given by
\begin{equation}\label{eq-21b}
    (\pont Rxy)_{[s,1]}=\frac{r^*(0,x;s,X_s)}{r^*(0,x;1,y)} \,R_{[s,1]}(\cdot\mid X_1=y)  .
\end{equation}
     \item[(3)] The forward and backward transition probability kernels of  $\pont Rxy$
satisfy for all $0\le s< t\le1$ and  $\bas Rst$-a.e.\
$(z,z'),$
\begin{eqnarray*}
  \pont rxy(s,z;t,dz') &=& \1_{\{r(s,z;1,y)>0\}} \frac{r(s,z;t,z')r(t,z';1,y)}{r(s,z;1,y)}\,m(dz') \\
  \haut rxy_*(s,dz;t,z') &=&
  \1_{\{r^*(0,x;t,z')>0\}} \frac{r^*(0,x;s,z)r^*(s,z;t,z')}{r^*(0,x;t,z')}\,m(dz)
\end{eqnarray*}
with the conventions $r(1,z;1,y)=\1_{\{z=y\}}$ and
$r^*(0,x;0,z)=\1_{\{z=x\}}.$
\end{enumerate}
\end{prop}

\begin{proof}
\boulette{(1)} Define  $
    \displaystyle{
    P^{\widetilde{xy}}:=\frac{ r(t,X_t;1,y)}{
    r(0,x;1,y)}R_{[0,t]}(\cdot\mid X_0=x) }
    $
and take a bounded nonnegative map $f$ and an event  $B\in\AAA 0t$. Then,
\begin{eqnarray*}
  E_{R}\left(P^{\widetilde{{xX_1}}}(B)f(X_1) \mid X_0=x \right)
 &=& \IX r(0,x;1,y) P^{\widetilde{xy}}(B)f(y)\,m(dy)\\
 &=& \IX E_{ R}[\1_B\, r(t,X_t;1,y)f(y)\mid X_0=x ]\,m(dy) \\
 &=& E_{ R} [\1_B\IX r(t,X_t;1,dy)f(y)m(dy) \mid X_0=x ]\\
  &=& E_R [\1_B\,E_R(f(X_1) \mid X_t)\mid X_0=x ]\\
  &=& E_R[\1_B\,E_R(f(X_1) \mid \XXX 0t)\mid X_0=x ] \\
  &=& E_R[\1_B\,f(X_1)\mid X_0=x ] \\
  &=& E_R[\pont Rx{X_1}(B)\,f(X_1)\mid X_0=x ]
\end{eqnarray*}
which proves  \eqref{eq-21f}. 

\Boulette{(2)} It is  analogous to  (1).

\Boulette{(3)} It is a direct corollary of  (1) and (2).
\end{proof}

\begin{examples}\label{ex-ponts}  Let us provide examples of several kinds of bridges.

\begin{itemize}
    \item[(i)]   The first example is standard. 
    Let $R=\bf W$ be a {\em Wiener \probmeas} on the set of real-valued continuous paths on  $[0,1]$,   with fixed initial 
    condition $x \in \R$.
   Assumption  $(\textrm{A}_1)$ is satisfied 
    since Brownian bridges can be constructed for 
    any $x,y \in \R$ (as Paul Lévy already proposed).
    Assumption $(\textrm{A}_2)$ is satisfied with  $m (dx) = dx$. 
   Then the forward and backward transition probability densities are given, for any $s\leq t, x,y \in \R $, by:
\begin{eqnarray*}
r(s,z;t,y)&=&\frac{1}{\sqrt{2 \pi (t-s)}} e^{-\frac{(y-z)^2}{2 (t-s)}}, \\
r^*(s,z;t,y)&=&\frac{r(0,x;s,z)r(s,z;t,y)}{r(0,x;t,y)  } .
\end{eqnarray*}
Therefore, due to (\ref{eq-21f}),  the Brownian bridge restricted to  $\AAA 0t$ satisfies 
\begin{equation*}
    ({\bf W}^{x,y})_{[0,t]}= \frac{1}{\sqrt{1-t}} e^{-\left(\frac{(y-X_t)^2}{2 (1-t)} - \frac{(y-x)^2}{2}\right)}\, {\bf W}_{[0,t]}
\end{equation*}
 \item[(ii)]  
  Let  $\Poi$ be the law of a  {\em Poisson process} with values in the set of  càdlàg  step 
  functions with positive unit jumps. Poisson bridges can be constructed for  all $x,y \in \R$ such that $y-x \in \N$.\\
Now suppose that $X_0$ under $\Poi$ is random, real-valued and admits a density: $\Poi_0(dx)=r_0(x) dx$ on  $\R$. 
As already remarked, such a process does not satisfy Assumption $(\textrm{A}_2)$. However  its dynamics is space- (and time-) homogeneous: 
$$
r(s,x;t,dy)= \delta_x *  r(0,0;t-s,dy)
$$
 and the transition kernel $r(0,0;u,dy)$ admits a  Poissonian density $r$ with respect to the counting measure $m$ on $\N$ : 
$$
r(0,0;u,dy) = r(u, y) \, m(dy) \quad \textrm{ where } \quad r(u,n) =  e^{-u} \, u^{n}/n! \, .
$$
 Therefore the proof of (\ref{eq-21f}) can be generalized to this case, since one  exhibits
 the density of the bridge, on the time interval   $[0,t]$, of the  Poisson process  
 between 0 and $n$ with respect to the standard Poisson process starting in 0. 
 Then, the density  on the time interval $[0,t]$ of the Poisson process  pinned at
 $x$ and $y$ with respect to the Poisson process starting from $x$ satisfies for  $\Poi_0$-a.e.\ $x$ and $y \in x+\N$,
\begin{eqnarray*}
    (\pont \Poi xy)_{[0,t]}&=& \frac{r(1-t,y-X_t)}{r(1,y-x)} \, \Poi_{[0,t]}(\cdot\mid X_0=x) \\
&=& e^t (1-t)^{y-X_t} \frac{(y-x)!}{(y-X_t)!}   \, \Poi_{[0,t]}(\cdot\mid X_0=x) .
\end{eqnarray*}
   
\item[(iii)]  
  Let  $\Ca$ be the law of a  {\em  Cauchy process} on $\Omega$. 
A regular version of Cauchy bridges can be constructed for  all $x,y \in \R$, see \cite{ChU11}.
  The forward transition density $r(s,x;t,y)$ is given, for each  $x,y \in \R $, by the Cauchy law with  parameter $t-s$ :
$$
r(s,x;t,y) = \frac{t-s}{\pi ((t-s)^2 + (y-x)^2)} 
$$
and for  $\Ca_0$-almost all  $x$,
\begin{equation*}
    (\pont \Ca xy)_{[0,t]}= (1-t) \, \frac{1 + (y-x)^2}{(1-t)^2 + (y-X_t)^2} \quad \Ca_{[0,t]}(\cdot\mid X_0=x) .
\end{equation*}
The computation of the density of the bridge on the time interval  $[s,1]$ follows the same schema, using the backward transition density and the initial value $\Ca_0$. 
One  could also consider the reversible situation, corresponding to $\Ca_0(dx)=dx$. 
This reversible measure cannot be normalized but the present techniques  remain valid  for $\sigma$-finite measures, see \cite{Leo12b}.
\item[(iv)] 
Several other examples of Lévy bridges can be found in \cite{PZ04}.
\end{itemize}
    \end{examples}

\section{Reciprocal \probmeas s} \label{sec-rpm}

We now enlarge our framework to the class of reciprocal \probmeas s. 
They are not necessarily Markov but they enjoy a more general time-symmetry which 
justifies their relevance in the study of quantum mechanical systems,
see \cite{Nel67,Na93,CZ08}. 
The dynamical properties of 
reciprocal diffusions were elaborated many years after their introduction by Bernstein, 
see \cite{Jam70,Z86} and Section \ref{JC} for further details.

In fact reciprocal processes can be viewed as  one-dimensional Markov random fields indexed by a compact time interval,
here $[0,1]$, see  Eq.\,\eqref{champMark}.
When the time parameter belongs to an unbounded interval, the issues 
are close to those that were encountered in the setting of Gibbs measures on the one-dimensional time space and quasi-invariant measures on the path space. They were extensively studied in the seventies in the context of Euclidean
quantum field theory, see \cite{CR75}, \cite{RY76}, \cite{DY78}  and references therein. 
A more recent result on existence and uniqueness of Gibbs measures for suitable potentials relative to Brownian motion was proved 
in \cite{OS99}. 

\subsection{Definition and basic properties}

Let us begin with the definition.
\begin{definition}[Reciprocal \probmeas]
A probability measure  $P$ on  $\Omega$ is called \emph{reciprocal} if for any times  $s\le u$ in  $[0,1]$ and for any events  $A\in\AAA 0s, B\in \AAA
    su, C \in \AAA u1$, see Figure \ref{fig-01},
    \begin{equation}\label{eq-rec}
    P(A\cap B\cap C \mid X_s, X_u)=P(A \cap C\mid X_s, X_u)P(B\mid X_s, X_u)\quad  P\pp 
    \end{equation}
    
\end{definition}

\begin{figure}
\includegraphics{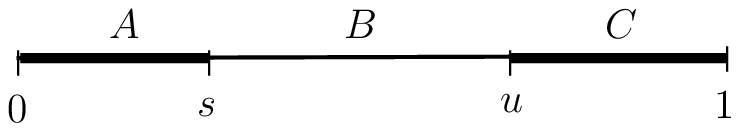}
\caption{}\label{fig-01}
\end{figure}

The above property, which was first formalized by Jamison in  \cite{Jam74}, states that under $P$, 
given the knowledge of the canonical process at  both times  $s$ and
$u$, the events ``inside'' $(s,u)$ and those ``outside'' $(s,u)$  are conditionally independent.
It is clearly  time-symmetric.
\\

Paralleling Theorem \ref{res-02}, we  present at Theorem \ref{rec-02} several characterizations of the reciprocal property. For the ease of the reader, we sketch its elementary proof.\\
Identity  (\ref{champMark}) states that a reciprocal \probmeas\  is indeed a Markov field indexed by  time, 
seen as a one-dimensional continuous parameter process. It means that conditioning an  event depending on the inside dat $X _{ [s,u]}$ by the knowledge of the outside data $X _{ [0,s]\cup[u,1]}$ amounts to simply conditioning it by the knowledge of the boundary data $(X_s,X_u).$ 

\begin{thm}\label{rec-02}Let $P$ be a probability measure  on $\Omega$.
Then the following assertions are equivalent:
\begin{itemize}
    \item[(1)] 
The \probmeas\ $P$ is reciprocal.
     \item[(1*)]  
The time-reversed \probmeas\   $P^*$ is reciprocal.
    \item[(2)]
For all  $0\le s\le u\le 1$ and all sets  $B\in\AAA su$,
\begin{equation}\label{champMark}
P(B\mid \XXX 0s, \XXX u1)=P(B\mid X_s,X_u) \quad  P\pp 
\end{equation}
\item[(3)]
For all  $0\le v \le r\le
s\le u\le1$ and all sets $A\in\AAA vr,$ $B\in\AAA su,$ see Figure \ref{fig-02},
\begin{equation*}\label{res-01}
    P(A\cap B\mid\XXX 0v, \XXX rs, \XXX u1)=P(A\mid X_v,X_r)P(B\mid
    X_s,X_u)\quad  P\pp  
\end{equation*}

\end{itemize}
\end{thm}

\begin{figure}
\includegraphics{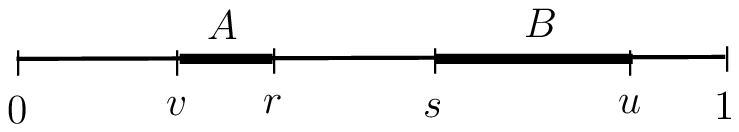}
\caption{}\label{fig-02}
\end{figure}

\begin{proof}
\Boulette{$(1)\Leftrightarrow (1^*)$}  Straightforward. 
\Boulette{$(1)\Rightarrow (2)$}
Let us take $ B \in\AAA su$. $P(B \mid \XXX 0s,\XXX u1) $ is the unique random variable  $\AAA 0s \vee \AAA u1$-measurable such that, 
for all $A\in\AAA 0s$ and $C \in\AAA u1$, 
$$
P(A\cap B\cap C ) = E[\1_A \1_C  P(B  \mid \XXX 0s,\XXX u1)].
$$
But, due to (\ref{eq-rec}), one has 
\begin{eqnarray*}
P(A\cap B\cap C ) = &=& E( P(A\cap B\cap C \mid X_s,X_u )) \\
  &=& E[ P(A\cap C \mid X_s,X_u ) P(B \mid X_s,X_u )]\\
  & =& E[E(\1_A \1_C P(B\mid X_s,X_u) \mid X_s,X_u)]\\
& =& E[\1_A \1_C P(B\mid X_s,X_u)].
  \end{eqnarray*}
This implies (2).
\Boulette{$(2)\Rightarrow (1)$}
Let us take  $0\le s\le u\le 1$, $A\in\AAA 0s, B \in\AAA su$, $C \in\AAA u1$ and  $f,g$ some measurable nonnegative functions. By definition, 
$$
  E[\1_A \1_B \1_C f(X_s) g(X_u)] = E [P(A\cap B\cap C \mid X_s,X_u)f(X_s) g(X_u) ]
$$
holds. But,
\begin{eqnarray*}
 E[\1_A \1_B \1_C  f(X_s) g(X_u)] &=&  E[E(\1_A \1_B \1_C  f(X_s) g(X_u)\mid \XXX 0s,\XXX u1) ] \\
  &=& E[\1_A \1_C  P(B  \mid X_s,X_u) f(X_s) g(X_u)]\\
  & =& E[P(A \cap C \mid X_s,X_u) P(B  \mid X_s,X_u) f(X_s) g(X_u)].
  \end{eqnarray*}
  Therefore
  $$
   P(A\cap B\cap C \mid X_s,X_u) =  P(A\cap C \mid X_s,X_u) P(B \mid X_s,X_u) .
  $$
 
\Boulette{$(2)\Rightarrow (3)$}
Take $A\in\AAA vr$ and $B\in\AAA su$. Then
\begin{eqnarray*}
 && P(A\cap B \mid\XXX 0v,\XXX rs,\XXX u1)\\
  &=& E\big[P(A\cap B \mid\XXX 0v,\XXX r1)\mid\XXX 0v,\XXX rs,\XXX u1\big] \\
   &\overset{\checkmark}{=}& E\big[P(A\mid X_v,X_r)\,\1_B\, \mid\XXX 0v,\XXX rs,\XXX u1\big] \\
  &=& E\Big[E\big(P(A\mid X_v,X_r)\,\1_B\, \mid \XXX 0s,\XXX u1\big)\mid\XXX 0v,\XXX rs,\XXX u1\Big] \\
   &\overset{\checkmark}{=}& E\big[P(A\mid X_v,X_r)P(B\mid X_s,X_u)\mid\XXX 0v,\XXX rs,\XXX u1\big] \\
  &=& P(A\mid X_v,X_r)P(B\mid X_s,X_u)
\end{eqnarray*}
where we used assumption (2) at the  $\checkmark$-marked equalities.
\Boulette{$(3)\Rightarrow (2)$}
It is enough to take  $A=\Omega$ and $v=t=s$.
\end{proof}

For a \probmeas, being Markov is stronger than being reciprocal. 
This  was noticed by Jamison  in  \cite{Jam70,Jam74}.

\begin{prop}\label{res-12}
Any Markov \probmeas\  is reciprocal, but the converse is false.
\end{prop}

\begin{proof}
Take $P$ a Markov \probmeas, $0\le s\le u\le 1$ and $A
\in\AAA 0s, B \in\AAA su$ and $C \in\AAA u1$. The following holds:
\begin{eqnarray*}
  P(A\cap B\cap C ) &=& E [P(A\cap B\cap C  \mid\XXX su)]\\
  &\overset{(i)}{=}& E[P(A \mid X_s)\1_B  P(C \mid X_u)] \\
  &=& E[P(A \mid X_s)P(B \mid X_s,X_u) P(C \mid X_u)]\\
  &\overset{(ii)}{=}& E[P(A \mid X_s)P(B \mid X_s,X_u) P(C \mid \XXX 0u)] \\
  &=& E[P(A \mid X_s)P(B \mid X_s,X_u) \1_C ] \\
  &\overset{(iii)}{=}& E[P(A \mid \XXX s1)P(B \mid X_s,X_u) \1_C ] \\
  &=& E[\1_A  P(B \mid X_s,X_u) \1_C ]
\end{eqnarray*}
Equality ({\it i}) is due to Theorem
\ref{res-02}(3).
Use the Markov property to prove ({\it ii}) and ({\it iii}). 
Therefore (\ref{champMark})  holds.\\
Examples \ref{ex-loop} and  \ref{exp-01}-(ii) below  provide examples of reciprocal measures which are not Markov.
Other counterexamples will be given in Section \ref{sec:MarkReciprocal}, where we point out 
 reciprocal processes whose endpoint marginal laws do not have the required structure which characterizes the Markov property.
\end{proof}

\begin{example}[Reciprocal measures on a loop space]\label{ex-loop}

Let us mention the following class of reciprocal -- but not Markov -- \probmeas s. 
Take a Markov \probmeas\  $R$  whose bridges are defined everywhere and $m$ any 
probability measure  on $\XX$.
Then
$$\Ploop:= \int_\XX \pont Rxx \, m(dx)$$
is a \probmeas\  that is concentrated on  loops,  i.e.\ paths such that $X_0=X_1$ almost surely, 
with both  initial and final marginal laws equal to  $m$. One can see this path measure as describing a periodic 
random process.
Due to Proposition  \ref{res-13}, $\Ploop$ is reciprocal with the mixing \probmeas\  
$\pi(dxdy) = m(dx) \delta_x(dy)$. 
\\
To see that the Markov property breaks down in general,  take $R ^{ xx}$ to be the usual Brownian bridge 
on $\R$ between $x$ and $x,$ choose $m=(\delta _{ -1}+ \delta _{ +1})/2$ and pick  any intermediate time $0<t<1.$ 
We have $\Ploop(X_1\ge 0\mid X _{
[0,t]})=\mathbf{1}_{ \left\{ X_0=+1\right\} }$, while 
\begin{eqnarray*}
\Ploop(X_1\ge 0\mid X_t)
&=&\Ploop(X_0=+1\mid X_t).
\end{eqnarray*}
In particular when $t=1/2,$ symmetry considerations lead us to $\Ploop(X_0=+1\mid X _{ 1/2})=1/2,$ 
implying that $\Ploop(X_1\ge 0\mid X _{ 1/2})=1/2\not= \textbf{1} _{ \left\{X_0=+1\right\} }=\Ploop(X_1\ge 0\mid X _{ [0,t]}).$ 

\end{example}

We will describe in a short while the typical structure of reciprocal \probmeas s.

\subsection{Pinning leads back to the Markov property} 

Proposition \ref{res-11bis}  states the stability of the Markov property by pinning. 
Similarly, Jamison pointed out in \cite[Lemma\,1.4 ]{Jam74} the remarkable 
property that pinning a reciprocal \probmeas, not only preserves its reciprocal property,
but also transforms it into a Markov one. 

\begin{prop}\label{res-11} \
Let $P$ be a reciprocal \probmeas. 
If either $X_0$ or $X_1$ is a.s. constant, then $P$ is a Markov \probmeas.
In particular any bridge $ P(\cdot \mid X_0, X_1) $, defined  $P_{01}$-almost surely, is a Markov \probmeas.
\end{prop}

\begin{proof}
Suppose $X_1$ is $P$-a.s. constant and take $0\leq s \leq u \leq 1$.
For any bounded measurable function $f$
\begin{eqnarray*}
 E_P (f(X_u) \mid X_{[0,s]}) &=& E_P (f(X_u) \mid X_{[0,s]}, X_1) = E_P (f(X_u) \mid X_s, X_1)\\
 &=& 
 E_P (f(X_u) \mid X_s)
\end{eqnarray*}
which characterises the Markov property of  $ P$ thanks Theorem \ref{res-02} (2). 
The case where $X_0$ is $P$-a.s. constant is solved by time reversal.
\end{proof}


\subsection{Mixing properly preserves the reciprocal property} 

To complement the previous subsection we analyse in which way  mixing (pinned) \probmeas s  perturbs their  reciprocal and/or Markov properties. \\
Mixing Markov \probmeas s  sometimes preserves the Markov property, but this is far from being the rule.
Similarly,   mixing reciprocal \probmeas s sometimes   results in a  reciprocal measure, but not always.
The following examples illustrate these assertions. Moreover, we construct in  (ii) an example of a reciprocal \probmeas\  which is not Markov.

\begin{examples}[Various mixtures of deterministic paths]\label{exp-01}  

Let $\XX=  \{ \textsf{a,b,c}\}$ be a  state space with three elements. 
We denote by $\delta_w$, $w \in \Omega$, the Dirac measure at the path $w.$ Any $\delta_w$  is Markov since the path  $w$ is
deterministic.

\begin{itemize}
    \item[(i)]  
One denotes by $\textsf{acb} \in \Omega$ the following path $w$:
$$
\textsf{acb}(t):= \1_{[0,1/3)}(t) \ \textsf{a}+ \1_{[1/3,2/3)}(t)\ \textsf{c} +\1_{[2/3,1]}(t)\ \textsf{b} .
$$ 
Similar notations are used for paths that only jump  at times 1/3 or 2/3.
\\
	The path measure
$$P = \frac{1}{4}(\delta_{\textsf{abc}}+\delta_{\textsf{aba}}+\delta_{\textsf{cba}}+\delta_{\textsf{cbc}})
$$ 
is the uniform  mixture of deterministic Markov paths and is Markov too.
 Indeed  $P_0 = \frac
    12(\delta_\textsf{a}+\delta_\textsf{c})$
 and the nontrivial transition \probmeas s which are given by 
$$
P(X_{1/3}=\textsf{b}\mid X_0 =  \textsf{a}) =  P(X_{1/3}=\textsf{b}\mid X_0 =  \textsf{c})=1
$$ 
and
\begin{eqnarray*}
P(X_{2/3}=\textsf{a}\mid X_{1/3} = \textsf{b}, X_0= \textsf{a})
&=&P(X_{2/3}=\textsf{a}\mid X_{1/3} = \textsf{b}, X_0= \textsf{c})\\&=& P(X_{2/3}=\textsf{c}\mid X_{1/3} = \textsf{b})=1/2
\end{eqnarray*}
entirely specify the dynamics of $P$.

    \item[(ii)] 
The path measure
$$
P=\frac{1}{2}(\delta_{\textsf{abc}}+\delta_{\textsf{cba}}),
$$
is reciprocal but not Markov. 
It is reciprocal since each boundary condition determines  the path. Nevertheless we observe that  $P$ is not Markov since 
$$
P(X_1=\textsf{a}\mid X_0=\textsf{a},X_{1/3}=\textsf{b})=0
$$
 while 
$$
P(X_1=\textsf{a}\mid X_{1/3}=\textsf{b})=1/2.
$$
    \item[(iii)] 
Now, we define paths with four states and three jumps at fixed times $1/4, 1/2$ et $3/4$,  such as 
$$
\textsf{abab}(t):= \1_{[0,1/4)}(t) \ \textsf{a}+ \1_{[1/4,1/2)}(t)\ \textsf{b} +\1_{[1/2,3/4)}(t)\ \textsf{a} + \1_{[3/4,1]}(t)\ \textsf{b}.
$$ 
    The path measure
    $P:=\frac{1}{2}(\delta_{\textsf{abab}}+\delta_{\textsf{cbcb}})$, 
    which is a mixture of reciprocal  paths (they are deterministic) is not  reciprocal anymore. Indeed 
$$
P(X_{2/3}=\textsf{a}\mid \XXX{0}{1/3},\XXX{4/5}{1})=
\mathbf{1}_{\{X_0=\textsf{a}\}}
$$
 while
$$
P(X_{2/3}=\textsf{a}\mid
X_{1/3},X_{4/5})=P(X_{2/3}=\textsf{a})=1/2.
$$
\end{itemize}
\end{examples}

\vspace{7mm}

To avoid the pathology (iii) let us now mix only measures 
 $\pont Rxy$ obtained as bridges of some given reciprocal measure $R$.

\begin{prop}\label{res-13}
Let $R$ be a reciprocal \probmeas\ such that  the mapping
$(x,y) \in \XtX \mapsto \pont Rxy$ is  defined everywhere and measurable. Then, for any probability measure   $\pi$ on $\XtX$,  the path measure  
\begin{equation*}
    P(\cdot)=\IXX \pont Rxy(\cdot)\,\pi(dxdy)
\end{equation*}
is reciprocal. Moreover, the bridges of  $P$ coincide with those of $R$,  $P$\pp 
\end{prop}

\begin{proof}
Let us show (\ref{champMark}) under $P$.
Let $0\le s\le t\le 1,$ $A\in\AAA 0s, B\in \AAA su$ and
$C\in\AAA u1$. Then,
\begin{eqnarray*}
  E_P[\1_A P(B\mid\XXX 0s,\XXX u1)\1_C]
  &=& P(A\cap B\cap C) \\
  &=& \IXX \pont Rxy(A\cap B\cap C)\,\pi(dxdy) \\
  &\overset{\checkmark}=&  \IXX E_{\pont Rxy}[\1_A R(B\mid X_s,X_t)\1_C]\,\pi(dxdy)\\
  &=& E_P[\1_A R(B\mid X_s,X_t)\1_C]
\end{eqnarray*}
where the reciprocal property is used at the marked equality.
Thus $P(B\mid\XXX 0s,\XXX t1)$ only depends on  $(X_s,X_t)$ and 
$$
P(B\mid\XXX 0s,\XXX t1)=R(B\mid X_s,X_t), \, P \pp,
$$ 
which completes the proof.
\end{proof}

Let us observe that this result does not contradict Example \ref{exp-01}-(iii). Indeed,  $P$ was expressed as a mixture of its own bridges, 
but not as a mixture of bridges of a given reciprocal \probmeas. It happens that there does not exist  any reciprocal \probmeas\  $R$ such that 
$\delta_{ \textsf{abab}}=R^{\textsf{ab}}$ and $\delta_{\textsf{cbcb}}=R^{\textsf{cb}}.$

\subsection{Reciprocal class associated with a \probmeas} \label{sec:recClass}

The previous proposition allows to construct classes of reciprocal \probmeas s based on some reference 
reciprocal \probmeas\  by varying the way of mixing  bridges. Therefore, we now recall the important concept of reciprocal class
which appears (implicitly)  in  \cite[\S3]{Jam74}, associated with a Markov reference \probmeas\  $R$ satisfying Assumptions (A) at page \pageref{AA}.

\begin{definition}[Reciprocal class associated with $R$]\label{def-clasrec} 
Suppose that $R$ is a reciprocal \probmeas\ such that
$(x,y) \in \XtX \mapsto \pont Rxy$ is   defined everywhere and measurable. The set of  probability measures on  $\Omega$ defined by 
 \begin{equation}\label{eq-12}
\Rec (R) :=\left\{ P=\IXX \pont Rxy(\cdot)\,\pi(dxdy); \pi \textrm{  probability measure  on } \XtX\right\}
\end{equation}
 is called the reciprocal class associated with $R$.
\end{definition}


In the case of a {\it discrete} state space $\XX$, the  hypothesis on $R$  becomes unnecessary.
One  should only make sure that the support of the mixing measure $\pi$ is included in the support of  $\bas R01$, in such a way that (\ref{eq-12}) makes sense.

In the particular case where $R$ is a Brownian diffusion defined on the space of continuous paths, the class $\Rec (R)$ can be characterized by two functions of
the drift of $R$, called reciprocal invariants. This was conjectured by  Krener in \cite{Kre88} and proved by Clark in \cite[Theorem\,1]{Cl91}. 
See also \cite{Th93} and \cite{CZ91} for their role in a second order stochastic differential equation
satisfied by the reciprocal class. 
Thereafter, Thieullen and the second author derived  an integration by parts formula on the path space that is expressed in terms of the reciprocal invariants of the Brownian diffusion $R$ and that  fully characterises the associated reciprocal class.
See  \cite{RT02} for one-dimensional diffusion processes and  \cite{RT05} for the multidimensional case.\\
When $R$ is a counting process (i.e.\  $\XX=\N$),  Murr provides a description of a reciprocal invariant associated 
with $\Rec (R)$,  as well as a characterisation of the reciprocal class through a duality formula, see \cite{Murr12,CLMR14}. 
An extension of this work for compound Poisson processes is done in \cite{CDPR14}, and 
to more general processes on graphs in \cite{CL14}. \\
For a recent review on stochastic analysis methods to characterize reciprocal classes, see \cite{Roe13}.

\subsection{Time reversal and reciprocal classes} 

We already saw in Theorem \ref{rec-02} that a path \probmeas\  is reciprocal 
if and only if its time-reversed is reciprocal too.
We now precise what is the image of a reciprocal class by time reversal. We give the proof, even if it looks rather natural.

\begin{prop}\label{res-renvrec}
Let $R$ be a reciprocal \probmeas\  as in Definition \ref{def-clasrec}. Then
$$	
P\in \Rec (R) \iff  P^*\in \Rec (R^*).
$$	
\end{prop}

We first prove the following auxiliary lemma.
\begin{lemma}\label{l213} Let $P$ be a probability measure on $\Omega$. \quad
\begin{enumerate}[(a)]
\item
Consider the diagram $\OO \overset{\Phi}{\to} \Phi(\OO)\overset{\theta}{\to} \YY$ where the mentioned sets and mappings are   measurable. Then, for any bounded measurable function $f:\Phi(\OO)\to\R$, we have
\begin{equation*}
E _{\Phi\pf P}(f| \theta)=\alpha(\theta)
\end{equation*}
 with  $\alpha(y):= E_P(f(\Phi)|\theta(\Phi)=y).$
\item
Consider the diagram $\YY\overset{\theta}{\leftarrow}\OO \overset{\Phi}{\rightarrow}\OO$ where the mentioned sets and mappings are   measurable. Suppose that $\Phi$ is one-to-one with measurable inverse $\Phi ^{-1}$.  Then,
\begin{equation*}
\Phi\pf\Big[P(\cdot\mid\theta=y)\Big]=\big[\Phi\pf P\big](\cdot\mid \theta\circ\Phi ^{-1}=y),\quad y\in \YY.
\end{equation*}
\end{enumerate}
\end{lemma}

\begin{proof}

 \boulette{(a)}
For any bounded measurable function $u:\YY\to\R$,
\begin{eqnarray*}
E _{\Phi\pf P}\Big[E _{\Phi\pf P}(f|\theta) u(\theta)\Big]
	&=& E _{\Phi\pf P}(f u(\theta))
	=E_P\Big[f(\Phi)u( \theta(\Phi))\Big]\\
	&=& E_P\Big[E_P(f(\Phi)|\theta(\Phi))\, u(\theta(\Phi))\Big]
	=E _{\Phi\pf P}(\alpha(\theta)\,u(\theta))
\end{eqnarray*}

\Boulette{(b)}
We add  a bounded measurable function $u$ to the diagram:\\
 $\YY\overset{\theta}{\leftarrow}\OO \overset{\Phi}{\rightarrow}\OO\overset{u}{\to}\R$ and compute, for $y\in\YY$,
\begin{eqnarray*}
E _{\Phi\pf P(\cdot\mid \theta=y)}(u)
	&=&E_P\big[u(\Phi)| \theta=y\big]\\&=&
	E_P\big[u(\Phi)|\theta\circ \Phi ^{-1}\circ\Phi=y\big]
	\overset{(i)}=E _{\Phi\pf P}(u| \theta\circ\Phi ^{-1}=y)
\end{eqnarray*}
where equality (i) is a  consequence of the above result (a).
\end{proof}

\begin{proof}[Proof of Proposition \ref{res-renvrec}]
In particular Lemma \ref{l213}-(b) implies that 
\begin{equation}\label{eq-27}
(R^{xy})^*=(R^*)^{yx},\quad \textrm{ for } R_{01}\pp \, x,y\in\XX.
\end{equation}

Let $P\in \Rec (R)$, then   $P(\cdot)=\IXX R ^{xy}(\cdot)\, P _{01}(dxdy).$ We now compute the integral of a function $u$ under  $ P^*$ :
\begin{eqnarray*}
E _{P^*}[u(X)]&=& E_P [ u(X^*)] =\IXX E _{(R^{xy})^*}(u)\, P _{01}(dxdy)\\
	&\overset{\eqref{eq-27}}{=}&\IXX E _{(R^*)^{yx}}(u)\, P _{01}(dxdy)=\IXX E _{(R^*)^{xy}}(u)\, (P^*)_{01}(dxdy).
\end{eqnarray*}
This means that   $ P^*(\cdot)=\IXX (R^*)^{xy}(\cdot)\,  (P^*)_{01}(dxdy)$, completing the proof of Proposition \ref{res-renvrec}.
\end{proof}

\subsection{Reciprocal subclass of dominated \probmeas s}\label{sec-FamRec} 

To make precise our structural analysis of reciprocal \probmeas s, 
we introduce a slightly smaller family of \probmeas s than the reciprocal class.
This subclass
only contains  \probmeas s which are dominated by the reference \probmeas\   $R$. 

\begin{definition}\label{def-02} 
Suppose that $R$ is a reciprocal \probmeas\   as in Definition \ref{def-clasrec}. 
We define the following set of probability measures on  $\Omega$:
 \begin{equation}\label{eq-12bis}
\Rec_{ac} (R) := \left\{ P: P=\IXX \pont Rxy\,\pi(dxdy);\  \pi\in\textrm{Proba}(\XX^2), \pi \ll \bas R01 \right\}\subset  \Rec (R) .
\end{equation}
\end{definition}

\begin{remarks}[about this definition] $\mbox{}$ 
\begin{enumerate}[(a)]
\item 
Due to Proposition \ref{res-13}, we notice that any element of $ \Rec_{ac} (R)$ is reciprocal.
\item 
We write $P\prec R$ when  $P$ disintegrates as in (\ref{eq-12bis}). 
Note that the  relation $\prec $ is transitive. But  it is not symmetric; this lack of symmetry arises when the marginal laws at time 0 and 1 are not equivalent in the sense of measure theory.
 Therefore  $\Rec_{ac} (R)$ is  not an equivalence class. If one wants to define a genuine equivalence relation
 $\sim$ between \probmeas s on $\OO$ one should assume that marginal laws at time 0 and 1 are equivalent. Then 
$P\sim R$ if and only if  $P\prec R$ and
$R\prec P$.
\end{enumerate}
\end{remarks}
 
As noticed in \cite{DY78} Proposition 3.5,  elements of $\Rec_{ac} (R)$ have a simple structure.

\begin{thm}\label{res-07} Each \probmeas\  $P$ in $\Rec_{ac} (R)$ is absolutely continuous with respect to  $R$  and satisfies
\begin{equation*}                                 
    P=\frac{d\pi}{d\bas R01}(X_0,X_1) \, R.
\end{equation*}
Conversely, if $P$ is a path measure defined by 
\begin{equation}\label{eq-PhR}
    P= h(X_0,X_1) \, R \,
\end{equation}
for some  nonnegative measurable function $h$ on $\XtX$, then $P\in \Rec_{ac} (R)$ and more precisely, $P$ is a  $\pi$-mixture of bridges of $R$ with 
$\pi (dx dy) := h(x,y) \, \bas R01 (dx dy)$.
\end{thm}

\begin{proof}
Let $P \in \Rec_{ac} (R)$ and $f$ any nonnegative bounded function. Due to Definition \eqref{eq-12bis}, 
\begin{eqnarray*}
  E_P(f)
  &=& \IXX E_R(f\mid  X_0=x,X_1=y) \frac{d\pi}{d\bas R01}(x,y)\, \bas R01(dxdy) \\
  &=& E_R\left(f \,\frac{d\pi}{d\bas R01}(X_0,X_1)\right),
\end{eqnarray*}
which proves the first assertion. 
For the second assertion, note that 
$$
 P(\cdot)=\IXX \pont Pxy(\cdot)\,\pi(dxdy) = \IXX h(x,y) \pont Rxy(\cdot)\,\bas R01(dxdy).
$$ 
This completes the proof of the theorem.
\end{proof}

The specific structure of $P$ which appears in  (\ref{eq-PhR}) can be regarded as a time-symmetric version 
of the  $h$-transform introduced by Doob  in \cite{Doob57}:
$$
P(\cdot \mid  X_0=x)= \frac{ h(x,X_1)}{c_x} \, R(\cdot \mid  X_0=x) \, \textrm{ for } R_{0}\pp \, x\in\XX
$$ 
and, in a symmetric way,
$$
P(\cdot \mid  X_1=y)= \frac{ h(X_0,y)}{C_y} \, R(\cdot \mid  X_1=y) \, \textrm{ for } R_{1}\pp \, y\in\XX.
$$ 

\subsection{Markov \probmeas s of a reciprocal class} \label{sec:MarkReciprocal}

Since the Markov property is more restrictive than the reciprocal property, 
it is interesting to describe the subset of $\Rec_{ac} (R)$ composed by the Markov measures. 
In other words, one is looking for the specific mixtures of \probmeas s which preserve Markov property.

If a \probmeas\  in $\Rec_{ac} (R)$ admits a density with respect to $R$ which is decomposable into a product
as in (\ref{eq-fg}), then it is Markov. Indeed this property is (almost) characteristic as will be stated below in Theorem \ref{res-20}. 
The first (partial) version of  Theorem \ref{res-20} can be found in   \cite[Thm.\,3.1]{Jam74}
when $R$ admits a strictly positive transition density: 
$P \in \Rec(R)$  if and only if there exist two  \probmeas s $\nu_0$ and $\nu_1$ on $\XX$  such that 
$$
\bas P01(dx dy) = r(0,x;1,y) \nu_0 (dx) \nu_1 (dy).
$$
Jamison  \cite[p.\,324]{Jam75}  commented on this structure as that of an ``$h$-path process in the sense of Doob''.
In the general framework of Markov field this result was proved in \cite[Thm.\,4.1]{DY78}.

Our statement emphasizes the role of  condition
\eqref{eq-x01} which, up to our knowledge,  comes out  for the first time.


\begin{thm}\label{res-20}
Let   $R$ and $P$ be two probability  measures on $\OO$ and suppose that $R$ is Markov. Consider the following assertions:
\begin{enumerate}
    \item[(1)] The \probmeas\  $P$ belongs to $\Rec_{ac} (R)$ 
     and is Markov.
    \item[(2)] There exist two measurable nonnegative functions $f_0$ and $g_1$  such that
 \begin{equation}\label{eq-39b}
    \dPR=\,f_0(X_0) g_1(X_1), \quad R\pp 
\end{equation}
\end{enumerate}
Then,  (2) implies assertion (1). \\
If we suppose moreover that there exists $0<t_0<1$ and a measurable subset $\XX_o\subset\XX$ such that 
 $R _{t_0}(\XX_o)>0$ and for all $z\in \XX_o$,
\begin{equation}\label{eq-x01}
R _{01}(\cdot) \ll R_{01}^{t_0z} (\cdot):= R((X_0,X_1)\in\cdot|X _{t_0}=z),
\end{equation}
  then  (1) and (2) are equivalent.
\end{thm}

\begin{proof}
 \boulette{$(2) \Rightarrow (1)$} It is contained in Example \ref{ex-2.16}. Note that Hypothesis \eqref{eq-x01} is not necessary.

    \Boulette{$(1) \Rightarrow (2)$} 
Since  $P$ is  Markov, Theorem \ref{res-19} applied with $t=t_0$ leads to
\begin{equation} \label{eq-albeta}
\dPR=\alpha(\XXX 0{t_0})\beta(\XXX {t_0}1) \quad  R\pp
\end{equation}
 with $\alpha$ and
$\beta$ two measurable nonnegative functions.  
But, since $P$ belongs to the reciprocal family of  $R$,
following Theorem \ref{res-07}, its  Radon-Nikodym derivative 
is
$$\dPR=h(X_0,X_1)$$ for some measurable nonnegative function $h$ on $\XX^2.$ This implies that
$$
\alpha(\XXX 0{t_0})\beta(\XXX {t_0}1)=h(X_0,X_1),\quad R\pp
$$
which in turns implies that the functions $\alpha$ and $\beta$ have the form
$$
\alpha(\XXX 0{t_0})=a(X_0,X_{t_0})\textrm{ and }
\beta(\XXX {t_0}1)=b(X_{t_0},X_1),\  R\pp
$$ 
with $a$ and $b$ two measurable nonnegative functions
 on $\XX^2.$ 
It follows that
$$a(x,z)b(z,y)=h(x,y) \quad \forall (x,z,y) \in \mathcal {N}^c\subset\XX^3,
$$
 where the set $\mathcal {N}\subset\XX^3$ is $R_{0,t_0,1}$-negligible. 
Now, with the notation
$$
\mathcal {N}_z:=\left\{(x,y);(x,z,y)\in \mathcal {N}\right\} \subset\XX^2,
$$
 we obtain
$$
0=R _{0,t_0,1}(\mathcal {N})
	=\IX R_{01}^{t_0z}( \mathcal{N}_z )\,R _{t_0}(dz)
$$
which implies that 
	$R_{01}^{t_0z}( \mathcal{N}_z) =0$ for $ R _{t_0}\pt \, z \in \XX_0$. Due to condition \eqref{eq-x01}, one deduces that there exists $z_o\in \XX_o$ such that  $R _{01}( \mathcal{N}_{z_o})=0.$ Taking  $f_0=a(\cdot, z_o)$ and $g_1=b(z_o,\cdot)$, we see that 
$$ 
h(x,y)=f_0(x)g_1(y), \quad \bas R01(dxdy)\pp,
$$ which proves that $dP/dR$ has the form  expressed in (\ref{eq-39b}).
\end{proof}

\begin{remarks}\label{rem-11} $\mbox{}$ 

\begin{enumerate}[(a)]
\item
Since  $R$ is Markov, condition \eqref{eq-x01} is equivalent to 
\begin{equation*}
  \forall z\in \XX_o, \quad R _{01}(\cdot)\ll R( X_0\in\cdot|X_{t_0}=z)\otimes R(X_1\in\cdot|X _{t_0}=z) .
\end{equation*}
\item
Without any additional condition on  $R$, both assertions of the above theorem fail to be equivalent. 
We provide a counter-example by constructing  a \probmeas\   $R$ which does not satisfy condition \eqref{eq-x01}
and a Markov \probmeas\  $P$ whose density with respect to  $R$ is not of the form \eqref{eq-39b}. \\
 Let  $R$ be the Markov \probmeas\  with state space $\XX=\left\{\textsf{a} , \textsf{b} \right\} $,  initial law
 $R_0=(\delta_\textsf{a} +\delta_\textsf{b} )/2$ and infinitesimal generator 
$ \begin{pmatrix}
     0&0\\
     \lambda & -\lambda
    \end{pmatrix} $ for some  $\lambda>0$.
The support of  $R$ is concentrated on two types of paths: the paths that are identically equal to $\textsf{a} $ or  $\textsf{b} $, and the other ones that  start from  $\textsf{b} $ 
with one jump onto  $\textsf{a} $ after an exponential waiting time in $(0,1)$ with law $\mathcal{E}(\lambda)$.  We see that  $R$ does not satisfy  \eqref{eq-x01}.
Indeed, for all  $t \in (0,1)$,
\begin{enumerate}
\item
$
R_{01}^{t\textsf{a}} ( \textsf{b},\textsf{b} )  =0$, but 
$R_{01} ( \textsf{b},\textsf{b})  = \frac{e^{-\lambda}}{2}>0.$ Thus, $R_{01} \not\ll R_{01}^{t\textsf{a}} .
$
\item
$
R_{01}^{t\textsf{b}} ( \textsf{a},\textsf{a})  =0$,  but  
$R_{01} ( \textsf{a},\textsf{a})  = \frac{1}{2}>0.$ Thus,  $R_{01} \not\ll R_{01}^{t\textsf{b }} .
$
\end{enumerate}
Consider the Markov \probmeas\  $P$ which gives half mass  to the deterministic constant paths equal to $\textsf{a} $ or   $\textsf{b} $.
It is dominated by  $R$ with density:
 $\displaystyle{\frac{dP}{dR}}=\left\{\begin{array}{ll}
1,& \textrm{if } X\equiv  \textsf{a} \\
e^\lambda,& \textrm{if } X\equiv \textsf{b} \\
0,& \textrm{if } X_0 \not = X_1   
\end{array}\right.$.
 This density  $dP/dR$ does not have the product form  \eqref{eq-39b}, since the system\\
$\left\{\begin{array}{lcl}
f (\textsf{a})g(\textsf{a})&=&1\\
f(\textsf{b })g(\textsf{b})&=&e^\lambda\\
f(\textsf{b })g(\textsf{a})&=&0
\end{array}\right.$
admits no  solution.
Remark that the functions  $\alpha$ and $\beta$ defined in  \eqref{eq-albeta} could be chosen as follows:
 $\alpha(X)=\beta(X)=1$ if $X\equiv\textsf{a}$,  $\alpha(X)=1$ if $X\equiv \textsf{b}$, $ \beta(X) = e^\lambda$ 
 if $X\equiv \textsf{b}$ and $\alpha(X)=\beta(X)=0$ otherwise. 
\end{enumerate}
\end{remarks}

\section{Reciprocal \probmeas s are solutions of  entropy minimizing problems}\label{sec-ent}

\newcommand{\ph}{\widehat{\pi}}
\newcommand{\Ph}{\widehat{P}}

We conclude this survey paper  going back to the problem that was originally addressed by  Schrödinger  in \cite{Sch31} and developed in  \cite{Sch32}.
It was the starting point of the theory of time-reversed Markov  \cite{Kol36} and reciprocal diffusion processes. A modern formulation of Schrödinger's
problem   is stated below at \eqref{Sdyn}.

 Motivated by a probabilistic solution of this problem, Bernstein \cite{Bern32} introduced the notion of reciprocal process. 
 It is likely that Bernstein wasn't aware of the fact that \eqref{Sdyn}'s solution  is not only reciprocal, but also Markov as was clearly demonstrated four
decades later  by Jamison  in \cite{Jam75}.

The new ingredient of this section is the relative entropy, or Kullback-Leibler divergence, introduced in \cite{KL51}. The relative entropy of a 
\probmeas\  $p$ with respect to another \probmeas\   $r$ on a measurable space  $\YY$ is given by 
\begin{equation*}
H(p|r):=\int_\YY \log \left(\frac{dp}{dr}\right) \,dp\in[0,+\infty]
\end{equation*}
when $p$ is dominated by $r$, and $+ \infty$ otherwise.

\subsection{Schrödinger's problem} \label{subsec:Schrproblem}

This problem is of a statistical physics nature.

\subsubsection*{Dynamical and static formulations of Schrödinger's problem}

Let us sketch some results which are presented in detail in the review paper \cite{Leo12e} (see also \cite{Ae96} and the references therein too).
The modern dynamical formulation  of  Schrödinger's problem  is as  follows.
Take a reference \probmeas\  $R$ on $\OO=D([0,1],\XX)$  and fix two probability measures $\mu_0$, $\mu_1$ on  $\XX$   (the marginal constraints). 
The aim is to minimize   $P \mapsto H(P|R) $ where $P$ varies in  the set of all path \probmeas s such that   $P_0=\mu_0$ and $ P_1=\mu_1$. 
A concise statement of Schrödinger's dynamical problem is
\begin{equation}\label{Sdyn}
H(P|R)\rightarrow \textrm{min};\qquad P\in \textrm{Proba}(\Omega): P_0=\mu_0, P_1=\mu_1
\tag{S$_{\mathrm{dyn}}$}
\end{equation}
Projecting via $(X_0,X_1)$ this variational problem onto the set $\XX^2$ of endpoint configurations, one obtains the following  associated  
static formulation: 
minimize $\pi \mapsto H(\pi|R _{01})$, where  $\pi$ is subject to vary in  the set of all probability measures on $\XX^2$ with prescribed marginals
$\pi_0(dx):=\pi(dx\times\XX)=\mu_0$ and $ \pi(dy):=\pi(\XX\times dy)=\mu_1$.  A concise statement of
Schrödinger's  static problem is
\begin{equation}\label{S}
H(\pi|R _{01})\rightarrow \textrm{min};\qquad \pi\in \textrm{Proba}(\XX^2): \pi_0=\mu_0, \pi_1=\mu_1
\tag{S}
\end{equation}
Let us recall the uniqueness result  \cite[Prop.\,2.3]{Leo12e} which was proved by F\"ollmer \cite{Foe85} in the special case of a Brownian
diffusion with drift.

\begin{prop}[]\label{res-Foellmer}
The dynamical and static  Schrödinger problems 
each admit at most one solution $\Ph$ and  $\ph$.
 If   $\Ph$ denotes the  solution of \eqref{Sdyn}, then  $\ph=\Ph _{01}$ is the solution of \eqref{S}. 
Conversely, if  $\ph$ solves  \eqref{S}, then the solution of  \eqref{Sdyn} is
\begin{equation}\label{eq-S04}
\Ph(\cdot)=\IXX R ^{xy}(\cdot)\,\ph(dxdy) \in \Rec_{ac}(R).
\end{equation}
\end{prop}

\begin{proof}[Sketch of the proof]
As   \emph{strictly} convex minimization  problems,  \eqref{Sdyn} and  \eqref{S} admit at most one solution.
Using the disintegration formula 
\begin{equation*}
H(P|R)=H(P _{01}|R _{01})+ \IXX H(P ^{xy}|R ^{xy})\,P _{01}(dxdy),
\end{equation*}
one obtains 
$
H(P _{01}|R _{01})\le H(P|R)
$
with equality  (when $H(P|R)<+\infty$) if and only if 
$
P ^{xy}=R ^{xy}
$
for $P _{01}$-almost all $(x,y)\in\XX^2$, which corresponds to $P \in \Rec_{ac}(R)$.
Thus, $\Ph$ is the solution of  \eqref{Sdyn} if and only if it disintegrates as \eqref{eq-S04}.
\end{proof}

\subsubsection*{The solution of  \eqref{Sdyn} is Markov.
}

We present an existence (and uniqueness) result for \eqref{Sdyn} and \eqref{S} which is proved  in   \cite{Leo12e}.

\begin{thm}\label{res-Ph}
Let  $R$ be a reference Markov \probmeas\ with identical
\footnote{This restriction is done for simplifying the statements. We mostly have in mind a stationary reference measure $R$.} marginal laws at time 0 and 1,
denoted by $m$.
 Suppose that  $R$  satisfies the following assumptions:
\begin{enumerate}
\item[(i)]
there exists $0<t_0<1$ and a measurable set $\XX_o\subset\XX$ such that  $R _{t_0}(\XX_o)>0$ and 
\begin{equation*}
R _{01}\ll R\big((X_0,X_1)\in\cdot|X _{t_0}=z\big),\quad \forall z\in\XX_o.
\end{equation*}
\item[(ii)]
there exists a nonnegative measurable function $A$ on  $\XX$ such that 
$$R _{01}(dxdy)\ge e ^{-A(x)-A(y)}\, m(dx)m(dy).$$ 
\end{enumerate}
Suppose also that the constraints $\mu_0$ and $\mu_1$ satisfy 
$$
H(\mu_0|m)+  H(\mu_1|m)<+\infty \textrm{ and }
\IX A\,d \mu_0 + \IX A\, d \mu_1<+\infty.
$$ 
Then \eqref{S} admits a unique solution  $\ph$. It satisfies
\begin{equation*}
\ph(dxdy)=f_0(x)g_1(y)\, R _{01}(dxdy)
\end{equation*}
for some  $m$-measurable  nonnegative functions $f_0,g_1:\XX\to[0,\infty)$ which solve the so-called  \emph{ Schrödinger system}:
\begin{equation} \label{eq-S13}
\left\{\begin{array}{lcll}
f_0(x)\, E_R[g_1(X_1)\mid X_0=x]&=&d \mu_0/dm (x),&\ \textrm{ for } m\pt \ x\\
g_1(y) \, E_R[f_0(X_0)\mid X_1=y]&=&d \mu_1/dm (y),&\ \textrm{ for } m\pt \ y.
\end{array}\right.
\end{equation}
Moreover, \eqref{Sdyn} admits the unique solution
\begin{equation}\label{eq-S15}
\widehat{P}=f_0(X_0)g_1(X_1)\,R.
\end{equation}
It inherits the  Markov property from $R.$
\end{thm}

\begin{remark}
In the Schrödinger system, $E_R[f_0(X_0)\mid X_1]$ and $E_R[g_1(X_1)\mid X_0]$ are well defined even if 
  $f_0(X_0)$ and $g_1(X_1)$ are not $R$-integrable. In fact,  $f_0$ and $g_1$ are measurable and  nonnegative; 
  therefore, only positive  integration is needed, see \cite{Leo12b}.
\end{remark}

Generalizing Proposition \ref{res-Foellmer}, we obtain without additional effort the following result.
\begin{cor}
Let $R$ be any reciprocal \probmeas.
The solution $\Ph$ of the variational problem  \eqref{Sdyn}, if it exists, belongs to the reciprocal family $\Rec_{ac}(R)$.
\end{cor}

\subsubsection*{A connection between Schrödinger's problem and PDEs } \label{JC} 

We give a PDE interpretation of the time-marginal flow $(\Ph_t)_{0\le t\le 1}$ of the solution $\Ph$ of \eqref{Sdyn}, with the aim of  clarifying its dynamical content.
Let us come back to Example \ref{ex-ponts}-(i), where $\XX=\R,$ $m(dx)=dx$, $R$ denotes the (unbounded\footnote{See \cite{Leo12b} for the technical modifications that are necessary to handle the case of an unbounded reference measure.}) reversible Wiener measure and $r(s,x;t,y)$ is its Gaussian kernel.
 Let us call $\rho_0(x) = \frac{d\mu_0}{dx}(x) $ and $\rho_1(y) = \frac{d\mu_1}{dy}(y) $.
 Then the system (\ref{eq-S13}) reduces to 
\begin{equation} \label{eq-S15bis}
\left\{\begin{array}{lcll}
f_0(x)\,\int r(0,x;1,y) g_1(y) \ dy &=& \rho_0(x) \\
g_1(y) \,\int f_0(x) r(0,x;1,y) \ dx &=& \rho_1(y) .
\end{array}\right.
\end{equation}
Schrödinger addressed the problem of the existence and uniqueness of solutions $(f_0,g_1)$ of this nonlinear system, given $r$ and the 
probabilistic boundary data 
$\rho_0$ and $\rho_1$. 
For $f_0$ and $g_1$ strictly positive, and $r$ considerably more general than the Gaussian kernel, 
introducing an entropy minimizing problem close to \eqref{S}, Beurling \cite{Be60} answered positively 
to this question. 
This was extended later by several authors, see \cite{FG97,Leo12e} for instance, taking advantage of the tight connection between \eqref{eq-S15bis} and
\eqref{S}. 
\\
Let us denote by $ f(t,z)$
 the solution of the parabolic initial value problem 
 \begin{equation} \label{eq-S16}
\left\{\begin{array}{ll}
(-\partial_t +\partial_{zz}^2 /2)f = 0, &0<t\le1 \\
f(0,\cdot)=f_0, &t=0
\end{array}\right.
\end{equation}
and by $g(t,z)$
 the solution of the adjoint final value problem 
\begin{equation} \label{eq-S17}
\left\{\begin{array}{ll}
(\partial_t +\partial_{zz}^2 /2)g = 0, &0\le t<1 \\
g(1,\cdot)=g_1, &t=1
\end{array}\right.
\end{equation}
Remark that $f(t,z)=E_R(f_0(X_0)\mid X_t=z)$ and $g(t,z)=E_R(g_1(X_1)\mid X_t=z).$
Thanks to the Markov property of $R$, Theorem \ref{res-02}-(3) entails that  for all $0\le t\le1,$ 
$$
\Ph_t(dz)=  f(t,z)g(t,z)\,dz.
$$
This relation is analogous to Born's formula: 
$$
\rho_t(dz)=\psi_t (z)\overline{\psi_t(z)}\,dz
$$ 
where $\rho$ is the probability of presence of the quantum particle and $\psi$ is the wave function. 
Indeed, as remarked in 1928 by the astrophysicist Eddington (this is quoted in \cite{Sch32}), 
the relation between  Schrödinger's equation and its  complex conjugate can be interpreted as time reversal. Therefore, $\psi_t$ and $\overline \psi_t$ can be interpreted as two wave functions
carrying respectively information from past and future. Indeed, they solve the standard quantum Schrödinger equation with respect to both directions of time.
Switching to the classical statistical physics problem \eqref{Sdyn}, one sees that the functions $f_t$ and $g_t$ share similar properties, replacing the
complex-valued Schrödinger equations in both directions of time  by the heat equations \eqref{eq-S16} and \eqref{eq-S17}.
This striking analogy was  Schrödinger's main motivation for introducing \eqref{Sdyn}. 
See \cite{Sch32} and also \cite{CZ08}, \cite[\S\,6,7]{Leo12e} for further detail.
\\
 Regarded as an element of $L^2(\R, dz)$ the solutions of (\ref{eq-S16}) and (\ref{eq-S17}) are analytic in the domain ${\mathcal Re} (t)>0$, continuous for 
 ${\mathcal Re} (t) \geq 0$ and their values on the imaginary axis respectively solve the (quantum mechanical)  Schrödinger equation and its complex conjugate.
 It is in this way that the Markov  measure $\Ph$ is a quantum-like \probmeas.
 The multiplicative structure of the density $d\Ph_t/dz$ appears as a stochastic analytic version of the complex conjugation of quantum functionals.
 When the Markov generator associated with $R$ is not self-adjoint, the same idea holds.  For (much) more on this PDE connection, see \cite{VZ12}. This quantum mechanical connection is the starting point of a stochastic deformation of classical mechanics \cite{Zam12}.

\subsection{A modification of Schrödinger's problem}

Having in mind these considerations  about Schrödinger's problem, it  appears that the notion of reciprocal measure was a technical intermediate 
step on the  way to the solution of \eqref{Sdyn}. Indeed,  Theorem \ref{res-Ph}  insures that   $\Ph$ is Markov, which is more specific than being reciprocal,
and its proof doesn't rely on the reciprocal property. Nevertheless, there exist  instances of  non-Markov reciprocal measures that are interesting in their own
right. Let us give a short  presentation of two problems relating entropy minimization and  reciprocal measures which are not Markov.

\subsubsection*{Reciprocal measures and entropy minimization}

Consider the following  modification of Schrödinger's problem
\begin{equation}\label{Spi}
H(P|R)\rightarrow \textrm{min};\qquad P\in \textrm{Proba}(\Omega): P _{01}=\pi
\tag{S$^{\pi}$}
\end{equation}
where $R$ is Markov and $\pi\in\textrm{Proba}(\XX^2)$ is given. Mimicking the sketch of proof of Proposition \ref{res-Foellmer}, 
it is easy to show that \eqref{Spi} admits a  solution  if and only if $H(\pi| R _{01})<\infty$ and that,  when this occurs,  this solution is unique and is
equal to
\begin{equation*}
\Rp(\cdot):=\IXX R ^{xy}(\cdot)\, \pi(dxdy).
\end{equation*}
When $\pi(dxdy)=f_0(x)g_1(y)\, R _{01}(dxdy)$
with $(f_0,g_1)$ solution of the Schrödinger system \eqref{eq-S13}, then $\eqref{Spi} = \eqref{Sdyn}$. By \eqref{eq-12bis}, we see that $\Rp$ belongs to the
reciprocal family $\Rec_{ac} (R)$ of $R.$ More precisely, when $\pi$ describes $\textrm{Proba}(\XX^2),$ defining 
$$
\Rec_H(R):=\left\{P: P \textrm{ solution of }\eqref{Spi} \textrm{ with } \pi\in \textrm{Proba}(\XX^2)\right\} ,
$$  
we see that
\begin{equation*}
\Rec_H(R)=\left\{\Rp;\quad  \pi\in\textrm{Proba}(\XX^2): H(\pi|R _{01})<\infty\right\} 
\end{equation*}
which is a little smaller than $\Rec_{ac}(R)$ for which $\pi$ is only required to satisfy $\pi\ll R _{01}.$ Notice that 
\begin{equation*}
\Rec_H(R)\subset\Rec_{ac}(R)\subset\Rec(R)
\end{equation*}
where these three classes are convex subsets of $\textrm{Proba}(\OO).$

\subsubsection*{Loop measures}

Example \ref{ex-loop} exhibits  a reciprocal loop measure $\Ploop=\IX R ^{xx}\,m(dx)$ which is   not Markov in general. 
Denoting $\pi_m(dxdy)=m(dx)\delta_x(dy)$, we see that $\Ploop=R ^{ \pi_m}.$

Remark that in the important case where $R$ is the reversible Brownian motion\footnote{See \cite{Leo12b} for the technical 
modifications that are necessary to handle the case of an unbounded reference measure.}, then $ \pi_m\not\ll R _{ 01}$ because $R _{ 01}(X_0=X_1)=0$ and
$\pi_m(X_0=X_1)=1.$ Consequently, (S$^{ \pi_m}$) has no solution. The endpoint constraint $\pi=\pi_m$ of \eqref{Spi}  is  degenerate  in the same way as $
(\mu_0,\mu_1)=(\delta_x,\delta_y)$ is a degenerate constraint of \eqref{Sdyn}. Indeed, both $\pi_m$ and $ \delta _{ (x,y)}$ verify $H(\pi_m|R _{ 01}),H(\delta
_{ (x,y)}|R _{ 01})< \infty$  and can be approximated by finite entropy constraints.

\subsubsection*{Stochastic representation of incompressible hydrodynamical flows}

Consider the following entropy minimization problem
\begin{equation}\label{eq-hydro}
H(P|R)\to \textrm{min};\qquad P\in \textrm{Proba}(\OO): P_t=m,\forall 0\le t\le 1, P _{01}=\pi
\end{equation}
which consists of minimizing the relative entropy $ H(P|R)$ of the path measure $P$ with respect to the Markov measure $R$ 
subject to the constraints that the time marginal flow $(P_t) _{ 0\le t\le 1}$ is constantly equal to a given  $m\in \textrm{Proba}(\XX)$ and that the endpoint
marginal $P_{ 01}$ is equal to a given $\pi\in \textrm{Proba}(\XX^2)$. This problem is a natural stochastization of Arnold's approach to the Euler equation for
incompressible fluids \cite{Arn66} which is connected to the Navier-Stokes equation. The justification of this assertion is part of a work in progress by two
of the authors. The incompressibility constraints is $P_t=m,\forall 0\le t\le 1$ when $m$ is the volume measure on the manifold $\XX$. The constraint $P
_{01}=\pi$ is Brenier's relaxation \cite{Bre89} of Arnold's final diffeomorphism. It can  be proved using the results of the present paper that  for a generic
endpoint constraint $\pi,$
the minimizer of \eqref{eq-hydro} (whenever it exists) is reciprocal but not Markov.

\section*{Acknowledgements}
The authors thank the referee for useful comments and especially for indicating the  reference \cite{DY78}.


\end{document}